\documentclass[11pt,oneside]{article}
\usepackage{supertabular, setspace,hyperref}
\usepackage[utf8]{inputenc}
\usepackage{amsmath}
\usepackage{mathtools}
\usepackage{amsfonts}
\usepackage{amsfonts,graphicx}
\usepackage{amssymb}
\usepackage[english]{babel}
\usepackage{graphicx}
\usepackage{amsthm}
\usepackage{float}
\restylefloat{figure}
\newtheorem{theorem}{Theorem}[section]

\newtheorem{prop}{Proposition}[section]
\newtheorem{example}{Example}[section]

\newtheorem{remark}{\sc Remark}
\newtheorem{lemma}{\sc Lemma}[section]
\newtheorem{corollary}{\sc Corollary}[section]
\newtheorem{definition}{\sc Definition}[section]

\newcommand{\be}{\begin{eqnarray}}
\newcommand{\ee}{\end{eqnarray}}
\newcommand{\Be}{\begin{eqnarray*}}
	\newcommand{\Ee}{\end{eqnarray*}}
\newcommand{\bee}{\begin{equation}}
\newcommand{\eee}{\end{equation}}
\newcommand{\ba}{\begin{array}}
	\newcommand{\ea}{\end{array}}
\newcommand{\bl}{\begin{lemma}}
	\newcommand{\el}{\end{lemma}}
\newcommand{\bd}{\begin{definition}}
	\newcommand{\ed}{\end{definition}}
\newcommand{\bt}{\begin{theorem}}
	\newcommand{\et}{\end{theorem}}
\newcommand{\bp}{\begin{proof}}
	\newcommand{\ep}{\end{proof}}
\newcommand{\bi}{\begin{itemize}}
	\newcommand{\ei}{\end{itemize}}
\newcommand{\br}{\begin{remark}}
	\newcommand{\er}{\end{remark}}
\newcommand{\bc}{\begin{corollary}}
	\newcommand{\ec}{\end{corollary}}
\newcommand{\bex}{\begin{example}}
	\newcommand{\eex}{\end{example}}
\usepackage{chngcntr}
\counterwithin*{equation}{section}
\counterwithin*{equation}{section}
\begin{document}
	\date{}
	\title{\textbf{Weighted quasi-metrics associated with Finsler metrics}}
	\maketitle
	\begin{center}
		\author{\textbf{Gauree Shanker, Sarita Rani}}
	\end{center}
	\begin{center}
		Department of Mathematics and Statistics\\
		School of Basic and Applied Sciences\\
		Central University of Punjab, Bathinda, Punjab-151 001, India\\
		Email: gshankar@cup.ac.in, saritas.ss92@gmail.com
\end{center}
\begin{center}
	\textbf{Abstract}
\end{center}
\begin{small}
The current paper deals with some new classes of Finsler metrics with reversible geodesics. We construct weighted quasi-metrics associated with these metrics. Further, we investigate some important geometric properties of weighted quasi-metric space. Finally, we discuss the embedding of  quasi-metric spaces with generalized weight.
\end{small}
\section{Introduction}
   The study of Finsler spaces with reversible geodesics is an interesting and important  topic in Riemann-Finsler geometry. A Finsler space is said to be with   reversible geodesics if for any of its oriented geodesic paths, the same path traversed  in the opposite sense is also a geodesic. Due to computational advantages, we consider this problem only for Finsler spaces with $(\alpha, \beta)-$metrics. So far, some progress has been done in this direction. Below we mention few of them.\\
   
   The conditions for a Randers space to have reversible geodesics have been obtained in \cite{Cra2005}. In \cite{Mas.Sab.Shi2010}, Masca et al. have obtained conditions for a Finsler space with $(\alpha, \beta)-$metric to have reversible geodesics and strictly  reversible geodesics. Further, in \cite{Mas.Sab.Shi2012}, they find necessary and sufficient conditions for two dimensional Finsler spaces with $(\alpha, \beta)-$metric to have reversible geodesics.  In \cite{Sab.Shi2012}, Sabau and Shimada study geometrical properties of Finsler spaces with reversible geodesics focusing especially on $(\alpha, \beta)-$metrics. In \cite{Sab.Shi.shi2014}, Sabau et al.  study metric structures associated to Finsler metrics and have obtained some important results.\\
   In this paper, first we construct some new classes of Finsler metrics with reversible geodesics, and then we find weighted quasi-metrics associated with these  Finsler metrics. 
 Next, we investigate some geometric properties of these spaces concerned with geodesic triangle. Finally, we show that every weighted quasi metric space with generalized weight can be viewed as the graph of a $1-$Lipschitz function over a suitable metric space. 
\section{Preliminaries}
Let $M$ be a connected, smooth  $n-$ manifold, and  $$ TM:= \bigsqcup_{x \in M}T_{x}M $$ be its  tangent bundle, where 
 $T_{x}M$ is  the tangent space at $x \in M.$
 A generic point $u$ of $TM$ is denoted by $u=(x,y)\in TM,$ where  $y\in T_{x}M.$ Locally, if $x=(x^i)$ is a local coordinate system on $M,$ then $u=(x,y)=(x^i,y^i)\in TM,$ where $y=y^i\left(\dfrac{\partial}{\partial x^i} \right)_x,$ \  $i=1,...,n.$\\
\begin{definition}
A Finsler metric on a smooth manifold $M$  is a function $F:TM\longrightarrow \left[ 0,\infty\right) $ with the following properties:
\begin{itemize}
	\item[(I)] \textbf{Regularity:} $F$ is $ C^{\infty}$ on the slit tangent bundle $ TM_0 := TM\backslash \left\lbrace 0\right\rbrace :=\left\{(x,y)\in TM\  | \ y \neq0 \right\}.$
	\item[(II)] \textbf{Positive homogeneity:} 
	$ F(x,\lambda y)=\lambda F(x,y) \; \forall \; \lambda >0. $
	\item[(III)] \textbf{Strong convexity:} The $n\times n$ Hessian matrix $\left( g_{ij}\right) =\left( \left[ \dfrac{1}{2}F^2 \right]_{y^iy^j} \right)  $ is positive-definite for all  $(x,y)\in TM_0 .$
\end{itemize}
\end{definition}
 If the condition $(II)$ is replaced with
\begin{itemize}
	\item[$(II)^{'}$] \textbf{Absolute homogeneity:} 
	$ F(x,\lambda y)=\lvert \lambda \rvert \  F(x,y) \ \  \forall \ \ \lambda \in \mathbb{R}, $
\end{itemize}
then the Finsler metric is called absolute homogeneous.\\
$(\alpha, \beta)-$ metrics are the most important Finsler metrics, which are written in the form $F=F(\alpha, \beta),$ where $F$ is positive one-homogeneous function of two arguments $\alpha$ and $\beta.$ Here $\alpha = \sqrt{a_{ij}(x) y^i y^j} $ is a Riemannian metric and  $\beta=b_i(x) y^i$ is a $1-$form on $M.$\\
In the present paper, we consider only $(\alpha, \beta)-$metrics constructed with positively defined Riemannian metric $(a_{ij})$ and $1-$form $\beta$ whose Riemannian length $b$ is less than $1,$ i.e., $b=\sqrt{b_{ij}(x)y^i y^j}<1.$ 
It is to be remaked that the Finsler metric is not absolute homogeneous, in general, due to the presence of one-form $\beta.$\\
Following Shen \cite{ChernShenRFG}, a Finsler $(\alpha, \beta)-$metric $F$ can be written in the following form
\begin{equation}
F=\alpha\ \phi(s),\ \ s=\dfrac{\beta}{\alpha},
\end{equation}
where $ \phi:I= [-r, r]\longrightarrow [0, \infty) $ is a $C^\infty$ function, and the interval $I$ can be chosen large enough such that $r\geq \dfrac{\beta}{\alpha}\ \forall\ x \in M $ and $y \in T_xM.$\\
Let us recall \cite{ChernShenRFG} the following lemma for later use:
\begin{lemma}{\label{gslem2.1}}(Shen's lemma)
The function  $F=\alpha \phi\left( s\right) \ (s=\beta/ \alpha) $ is a Finsler metric for any Riemannian metric $\alpha=\sqrt{a_{ij}y^i y^j}$ and $1-$form $\beta=b_i(x) y^i$ with $ \lVert\beta_x\rVert_\alpha <b_0 $ if and only if $\phi=\phi(s)$ is a positive $C^\infty$ function on $ (-b_0, b_0) $   satisfying the following conditions:
\begin{equation}{\label{2.2}}
\phi(s)-s\phi^{'}(s)+\left( b^2-s^2\right)\phi^{''}(s)>0, \ \ \lvert s\rvert\leq b<b_0,
\end{equation}
\begin{equation}{\label{2.3}}
\phi(s)-s\phi^{'}(s)>0, \ \ \lvert s\rvert<b_0.
\end{equation}
\end{lemma}
There are so many important  $(\alpha, \beta)-$ metrics constructed till now e.g. Randers metric  $F=\alpha+\beta,$ Matsumoto metric  $F=\dfrac{\alpha^{2}}{\alpha-\beta}$ etc. (see \cite{GHAM1903} for more examples).
\begin{definition}
	Let $\gamma :[a,b]\longrightarrow M$ be a piecewise smooth curve on a Finsler manifold $M,$ then its Finslerian length is defined as
	\begin{equation}
	L_F(\gamma) :=\int_{a}^{b}F\left( \gamma(t), \dot{\gamma}(t)\right)  dt.
	\end{equation}
\end{definition}
The Finsler metric $F$ induces a Finslerian distance function $d_F :M \times M \longrightarrow [0,\infty),$ defined by 
\begin{equation}
d_F(p,q)= \inf_\gamma L_F(\gamma),\ \ \ p,q \in M,
\end{equation} 
where the infimum is taken over all piecewise smooth curves $\gamma$ on $M$ joining the points $p,q\in M.$ It is to be remarked that $d_F $ is not symmetric, i.e., $d_F(p,q)\neq d_F(q,p),$ in general. But, if $F$ is absolute homogeneous, then the Finslerian distance $d_F$ becomes symmetric. \\
Recall \cite{BCS} that a smooth curve $\gamma :[0,1]\longrightarrow M$ is called a geodesic of $(M,F)$ if it minimizes the Finslerian length for all piecewise smooth curves that keep their end points fixed.
\begin{definition}
Let $F :TM \longrightarrow[0,\infty)$ be a Finsler metric, then its reverse Finsler metric denoted by $\bar{F},$ is a map  $\bar{F} :TM \longrightarrow[0,\infty),\ $ given by $ \bar{F}(x,y)=F(x, -y).$ 
\end{definition}
Next, we have the following lemma:
\begin{lemma}
	If $\gamma(t)$ is a geodesic of the Finsler metric $F,$ then $\bar{\gamma}(t) :=\gamma(1-t)$ is a geodesic of $\bar{F},$ but not necessarily a geodesic of $F$ in general.
\end{lemma}
 Let us recall (\cite{Cra2005}, \cite{Mas.Sab.Shi2010}) the following definitions and results  for later use:
 \begin{definition}
 	A Finsler metric $F$ is called with reversible geodesic if and only if for any geodesic $\gamma(t)$ of $F,$ the reverse curve $\tilde{\gamma}(t) :=\gamma(1-t)$ is also a geodesic of $F.$
 \end{definition}
It is to be noted that in general $d_F(p,q)=d_{\bar{F}}(q,p)$ for all $p, q$ in $M,$ but even though a Finsler metric is with reversible geodesic it does not mean that it has symmetric distance function, except for the absolute homogeneous case. In this connection, we have
\begin{theorem}
	Let $d_F :M \times M \longrightarrow [0,\infty)$ be a distance function defined on a connected and complete Finsler manifold $(M,F).$ Then the necessary and sufficient condition for $d_F$ to be symmetric is that $F$ is absolute homogeneous, i.e., $F(x,y)=F(x,-y)=\bar{F}(x,y).$ 
\end{theorem}

\section{$(\alpha, \beta)-$ metrics with reversible geodesics}
It is well known that a smooth curve $\gamma :[a,b]\longrightarrow M$ is a constant Finslerian speed geodesic of the Finsler space $(M,F)$ if and only if it satisfies:
\begin{equation}{\label{gs3.1}}
\ddot{\gamma}^i(t)+2G^i\left( \gamma(t), \dot{\gamma}(t)\right) =0,\ i=1,...,n,
\end{equation}
where the functions $$ G^i :TM_0 \longrightarrow \mathbb{R} $$ are given by
\begin{equation}{\label{gs3.2}}
G^i(x,y)=\Gamma^i_{jk}(x,y) y^j y^k,
\end{equation}
with $$ \Gamma^i_{jk}(x,y)=\dfrac{1}{2}\  g^{is} \left( \dfrac{\partial \mathnormal{g}_{sj}}{\partial x^k}+ \dfrac{\partial \mathnormal{g}_{sk}}{\partial x^j}- \dfrac{\partial \mathnormal{g}_{jk}}{\partial x^s} \right).  $$
The vector field $ \Gamma := y^i\dfrac{\partial}{\partial x^i}-2G^i \dfrac{\partial}{\partial y^i} $ with $G^i$ given in (\ref{gs3.2}), is a well defined vector field on $ TM $ whose integral lines are the canonical lifts $ \tilde{\gamma}(t)= \left( \gamma(t), \dot{\gamma}(t)\right) $ of the geodesics satisfying (\ref{gs3.1}). Due to this fact, the vector field $\Gamma$ is called the canonical geodesic spray of the Finsler space $(M,F)$ and $G^i$ are called the coefficients of the geodesic spray $\Gamma.$\\
\begin{definition}{\label{defpe}}
Let $F$ and $\bar{F}$ be two different Finsler metrics on the same manifold $M$ with corresponding geodesic sprays  $ \Gamma := y^i\dfrac{\partial}{\partial x^i}-2G^i \dfrac{\partial}{\partial y^i} $ and $ \bar{\Gamma} := y^i\dfrac{\partial}{\partial x^i}-2\bar{G}^i \dfrac{\partial}{\partial y^i} $ respectively. Then $F$ and $\bar{F}$ are said to be projectively equivalent if their geodesics coincide as set points. 
\end{definition}
Recall (\cite{Cra2005}, \cite{Mas.Sab.Shi2010}) that Euler-Lagrange equation of a Finsler space $(M,F)$ can be written in terms of the geodesic spray as $ \Gamma\left(\dfrac{\partial F}{\partial y^i} \right)-\dfrac{\partial F}{\partial x^i}=0.  $ Therefore, 
in terms of geodesics spray, the condition for projective equivalence given in  definition  (\ref{defpe}) is equivalent to $ \Gamma\left(\dfrac{\partial F}{\partial y^i} \right)-\dfrac{\partial F}{\partial x^i}=0. $\\
Therefore, the notion of reversible geodesics can be re-defined in the following manner:
\begin{definition}
A Finsler space $(M,F)$ is with reversible geodesics  if and only if $F$ and its reverse function $\bar{F}$ are projectively equivalent, i.e., the geodesics of  $F$ and $\bar{F}$  coincide as set points. If $ \bar{\Gamma}$ is the reverse  geodesic spray, then $F$ is with reversible geodesics  if and only if $ \bar{\Gamma}\left(\dfrac{\partial F}{\partial y^i} \right)-\dfrac{\partial F}{\partial x^i}=0.  $
\end{definition}
\begin{definition}
A Finsler space $(M,F)$ is called (locally) projectively flat if all its geodesics are straight lines. An equivalent condition is that the spray coefficients $G^i$ of $F$ can be expressed as $G^i=P(x,y)y^i,$ where $P(x,y)=\dfrac{1}{2} \dfrac{\partial F}{\partial x^k} y^k.$
\end{definition}
In \cite{GHAM1903}, Hamel has given an equivalent characterization of projective flatness of Finsler metric $F$ by the relation 
$$ {F}_{x^my^k}y^m-{F}_{x^k}=0. $$
Let us recall \cite{Sab.Shi2012} the following:
\begin{theorem}{\label{gsthm3.1}}
	Let $F=F_0+ \epsilon \beta$ be a Randers changed Finsler metric, where $F_0$ is an absolute homogeneous $(\alpha, \beta)-$metric. Then, any two of the following properties imply the third one:
\begin{enumerate}
	\item[(i)] $F$ is projectively flat.
	\item[(ii)] $F_0$ is projectively flat.
	\item[(iii)] $ \beta$ is closed.
\end{enumerate}
\end{theorem}
By using Hamel's relation, it can be easily proved that if $F$ is projectively flat, then it is with reversible geodesics, but  not conversely, which can be seen by the examples given in \cite{She2009} and \cite{SheYil2008}.
Further, by using theorem (\ref{gsthm3.1}), it can be easily seen that any of the following $(\alpha, \beta)-$metrics are with reversible geodesics if and only if $\beta$ is a closed $1-$form on $M$: 
\begin{enumerate}
	\item[(a)] 
	$$F(\alpha, \beta)=a_0 \alpha + \sum_{k=1}^{p} a_{2k}\ \dfrac{\beta^{2k}}{\alpha^{2k-1}} + \epsilon \beta,$$
	$ \text{where\ }  a_0, a_{2k}, \epsilon\   \text{are non-zero constants} ,\ p\in \mathbb{N}. $
	\item[(b)] Randers metric $ F(\alpha, \beta)=\alpha+ \beta.$
	\item[(c)] Square metric $ F(\alpha, \beta)=\dfrac{(\alpha+ \beta)^2}{\alpha}.$
\end{enumerate}
Next, we recall \cite{Sab.Shi2012} the following:
\begin{theorem}{\label{gsthm3.2}}
Let $M$ be a non-Riemannian $n(\geq 2)-$dimensional Finsler space with $(\alpha, \beta)-$metric $F$, which is not absolute homogeneous. Then $F$ is with reversible geodesics if and only if 
$$ F(\alpha, \beta)=F_0(\alpha, \beta)+\epsilon \beta, $$ where $F_0$ is absolute homogeneous, $\epsilon$ is non-zero constant and $\beta$ is a closed $1-$form on $M.$	
\end{theorem}
By using theorem (\ref{gsthm3.2}), it can be easily proved that the $(\alpha, \beta)-$metrics $F(\alpha, \beta)=\alpha\ \phi(s),$ with $\phi(s)$ given by
\begin{enumerate}
	\item[(a)] $ \phi(s)=s^2+s+1, $
	\item[(b)] $ \phi(s)=s^2+2s+1, $
	\item[(c)] $ \phi(s)=cos s+as $\\
	are with reversible geodesics.
\end{enumerate}
Next, we construct some more examples of $(\alpha, \beta)-$metrics with reversible geodesics.
\begin{theorem}{\label{gdthm3.3}}
	Let $F=\alpha\ \phi(s)\ \left( s=\beta/\alpha\right),$ where $\phi :[-r,r]\longrightarrow[0,\infty)$ is a  smooth function. Then in either of the following cases, 
	$F$ is an 	$(\alpha, \beta)-$metric.
	\begin{enumerate}
		\item[(i)] $ \phi(s)=e^s+s+1.$
		\item[(ii)] $ \phi(s)=s^2+2s+2. $ 
		\item[(iii)] $ \phi(s)=s^2+s+2. $ 
		\item [(iv)] $ \phi(s)=(s+2)^2. $
	\end{enumerate}
\end{theorem}
\begin{proof}
	In either case, it can be easily seen that the function $\phi$ satisfies the conditions (\ref{2.2}) and (\ref{2.3}) of Shen's lemma (\ref{gslem2.1}), which is a criterion for $F=\alpha \phi(s)$ to be an $(\alpha, \beta)-$metric. Therefore, in either case, $F$ is an   $(\alpha, \beta)-$metric.
\end{proof}
\begin{theorem}{\label{gsthm3.4}}
	Let $\left( M,F(\alpha, \beta)\right) $ be a non-Riemannian $n(\geq 3)-$dimensional Finsler space with $\beta,$ a closed $1-$form.  Then in either of the following cases, $F=\alpha \phi(s)$ is with reversible geodesics.
\begin{enumerate}
	\item[(i)] $ \phi(s)=s^2+2s+2. $ 
	\item[(ii)] $ \phi(s)=s^2+s+2. $
	\item [(iii)] $ \phi(s)=(s+2)^2. $ 
\end{enumerate}	
\end{theorem}
\begin{proof}
\begin{enumerate}
	\item[(i)] 
We have	$ \phi(s)=s^2+2s+2, $\\
therefore,
\begin{align*}
F(\alpha, \beta)
&=\dfrac{\beta^2}{\alpha} +2\beta+2\alpha\\
&=F_0(\alpha, \beta) +2\beta,\ \ \text{where}\ \ F_0(\alpha, \beta)=\dfrac{\beta^2}{\alpha} +2\alpha
\end{align*} 
It can be easily seen that $F_0$ is absolute homogeneous. Then by theorem (\ref{gsthm3.2}), $F$  is with reversible geodesics.
\item[(ii)]
We have	 $ \phi(s)=s^2+s+2, $\\
therefore,
\begin{align*}
F(\alpha, \beta)
&=\dfrac{\beta^2}{\alpha} +\beta+2\alpha\\
&=F_0(\alpha, \beta) +\beta,\ \ \text{where}\ \ F_0(\alpha, \beta)=\dfrac{\beta^2}{\alpha} +2\alpha
\end{align*} 
Clearly, $F_0$ is absolute homogeneous. Then by theorem (\ref{gsthm3.2}), $F$  is with reversible geodesics.	
\item [(iii)] We have $ \phi(s)=(s+2)^2=s^2+4s+4 $ \\
therefore, 
\begin{align*}
F(\alpha, \beta)
&=\dfrac{\beta^2}{\alpha}+4(\alpha +\beta)\\
&=F_0(\alpha, \beta) +4\beta,\ \ \text{where}\ \ F_0(\alpha, \beta)=\dfrac{\beta^2}{\alpha}+4\alpha
\end{align*} 
Clearly, $F_0$ is absolute homogeneous. Therefore, by theorem (\ref{gsthm3.2}), $F$  is with reversible geodesics.
\end{enumerate}
\end{proof}
\section{Weighted quasi-metrics associated with Finsler metrics}
It is well known that Riemannian spaces can be represented as metrics spaces. It means that if $(M,\alpha)$ is a Riemannian space, then one can induce a metric $d_\alpha : M\times M \longrightarrow [0,\infty)$ on $M,$ given by\\
\begin{equation}{\label{gs4.1}}
d_\alpha(x,y) := \inf_{\gamma \in \Gamma_{xy}} \int_{a}^{b} \alpha \left( \gamma(t), \dot{\gamma}(t)\right)  dt, 
\end{equation}
where
$$ \Gamma_{xy} := \bigg\{ \gamma :[a,b] \longrightarrow M\  | \ \ \gamma\ (\text{piecewise smooth curve}),\ \gamma(a)=x,\ \gamma(b)=y \bigg\}$$
is a set of curves joining $x$ and $y,$\  $ \dot{\gamma}(t) :=\dfrac{d \gamma(t)}{dt}$ the tangent vector to $\gamma$ at $\gamma(t),$ and $\alpha(x,X)$ the Riemannian norm of the vector $X \in T_xM.$  One can easily verify that $d_\alpha$ is a metric on $M,$ i.e., it satisfies the following properties:
\begin{enumerate}
	\item [(a)] \textbf{Positiveness:} $ d_\alpha(x,y)>0 $ if $x \neq y,  \ d_\alpha(x,x)=0,$ 
	\item [(b)] \textbf{Symmetry:} $d_\alpha(x,y)=d_\alpha(y,x),$
	\item [(c)]	\textbf{Triangle inequality: } $ d_\alpha(x,y) \leq d_\alpha(x,z)+ d_\alpha(z,y) $\\
	for any  $x,y,z \in M.$
\end{enumerate}
The manifold $M,$  equipped with metric $d_\alpha,$ i.e., $(M, d_\alpha)$ is called the induced metric space.\\
One knows that Finsler structures are more general structures than Riemannian structures. Therefore, similar to Riemannian case, it is natural to induce a metric  $d_F$ on a Finsler space $(M,F),$ given by 
\begin{equation}{\label{gs4.2}}
d_F : M\times M \longrightarrow [0,\infty),\ \ d_F(x,y) := \inf_{\gamma \in \Gamma_{_{xy}}} \int_{a}^{b} F \left( \gamma(t), \dot{\gamma}(t)\right)  dt. 
\end{equation}
But unfortunately, in this case $d_F$ is not symmetric, i.e., $d_F(x,y) \neq d_F(y,x)$ in general, unlike the Riemannian case. Here, we find that $d_F$ is a special case of quasi-metric space, which is defined as follows \cite{Kii.Vaj1994}:
\begin{definition}
	Let $M$ be any set. Then a function $d : M\times M \longrightarrow [0,\infty)$ is a quasi-metric on $M,$ if it satisfies the following axioms:
	\begin{enumerate}
		\item [(a)] \textbf{Positiveness:} $ d(x,y)>0 $ if $x \neq y,  \ d(x,x)=0,$ 
		\item [(b)]	\textbf{Triangle inequality: } $ d(x,y) \leq d(x,z)+ d(z,y), $ 
		\item [(c)] \textbf{Separation axiom:} $d(x,y)=d(y,x)=0\ \implies\ x=y$ \\
		for any  $x,y,z \in M.$
	\end{enumerate}
The space equipped with quasi metric is called quasi-metric space. 
\end{definition}
Here, it is to be noted that every metric is a quasi-metric but not conversely. A quasi-metric $d$ that satisfies the symmetry condition, i.e., $d(x,y)=d(y,x)$
 for all $x,y \in M,$ becomes a metric on $M.$
 It can be easily seen that in case of Riemannian and absolute homogeneous Finsler spaces, every quasi-metric is a metric on $M.$\\
One special class of quasi-metric spaces are the so called weighted quasi-metric spaces $(M,d,\omega),$ where $d$ is quasi-metric on $M$ and for each $d,$ there exist a function $\omega : M \longrightarrow[0,\infty),$ known as weight of $d,$ that satisfies 
\begin{enumerate}
	\item [(d)] \textbf{Weightability:} $d(x,y) +\omega(x)=d(y,x)+ \omega(y) \ \forall \ \ x,y\  \in M.$
\end{enumerate}
If the weight function $\omega$ is defined from $ M$ to $\mathbb{R},$ then $\omega$ is known as generalized weight. A generalized weighted quasi-metric space is just a weighted quasi metric space in which both the quasi metric and the weight function are not necessarily non- negative.\\
The notion of weighted quasi-metric spaces was first introduced by Matthews \cite{Mat1993} in the context of theoretical computer science. In \cite{Kii.Vaj1994}, K$\ddot{u}$nzi and Vajner have studied the topological properties of weighted quasi-metric spaces in detail.
\begin{theorem}{\label{gsthm4.1}}
Let $(M,F)$ be a smooth, simply connected Finsler $n-$manifold, where $F$ is either of the following:
\begin{enumerate}
	\item[(i)]  $F(\alpha, \beta)=\dfrac{(\alpha+\beta)^2}{\alpha},$
	\item[(ii)]  $F(\alpha, \beta)=\alpha \cos \left(\dfrac{\beta}{\alpha}\right)+a\beta,$
	\item[(iii)]  $F(\alpha, \beta)= \alpha \left( \dfrac{\beta}{\alpha} +2\right) ^2,$
	\item[(iv)]  $F(\alpha, \beta)=\alpha \left( \dfrac{\beta^2}{\alpha^2} +\dfrac{\beta}{\alpha}+2\right) ,$
	\item [(v)] $F(\alpha, \beta)=\alpha \left( \dfrac{\beta^2}{\alpha^2} +\dfrac{2 \beta}{\alpha}+2\right).$
\end{enumerate}	
Then, for an exact $1-$form $\beta,$ in either case, $F$ induces a weighted quasi-metric $d_F$ on $M.$ 
\end{theorem}
\begin{proof}
\begin{enumerate}
\item[(i)] 
	We have  
\begin{align*}
F(\alpha, \beta)
&=\dfrac{(\alpha+\beta)^2}{\alpha}\\
&=F_0(\alpha, \beta) +2\beta,
\end{align*}
where\ \ $F_0(\alpha, \beta)= \alpha+\dfrac{\beta^2}{\alpha},\ \alpha = \sqrt{a_{ij}(x) y^i y^j} $ is a Riemannian metric and  $\beta=b_i(x) y^i$ is an exact $1-$form. \\
Clearly $F_0$ is absolute homogeneous, and hence by theorem (\ref{gsthm3.2}), $F$ is with reversible geodesics.\\
Let $\gamma_{_{xy}} \in \Gamma_{_{xy}}$ be an $F-$geodesic, which is in the same time an $F_0-$geodesic, then from equation (\ref{gs4.2}), we have
\begin{equation}{\label{gs4.3}}
\begin{split}
d_F(x,y)
&= \int_{a}^{b} F \left( \gamma_{_{xy}}(t), \dot{\gamma}_{_{xy}}(t)\right)  dt\\
&= \int_{a}^{b} F_0 \left( \gamma_{_{xy}}(t), \dot{\gamma}_{_{xy}}(t)\right)  dt\ +2  \int_{a}^{b} b_i \left( \gamma_{_{xy}}(t)\right) \dot{\gamma}^i_{_{xy}}(t) dt\\ 
&= \int_{\gamma_{_{xy}}} \left( \alpha+\dfrac{\beta^2}{\alpha}\right) +2 \int_{\gamma_{_{xy}}} \beta. 
\end{split}
\end{equation}	
For a fixed point  $a \in M,$ let us define a function $\omega_a : M \longrightarrow \mathbb{R}$ given by $ \omega_a(x)=d_F(a,x)-d_F(x,a).$ Then from equation (\ref{gs4.3}), it follows that 
\begin{equation*}
\omega_a(x)= \int_{\gamma_{_{ax}}} \left( \alpha+\dfrac{\beta^2}{\alpha}\right) +2 \int_{\gamma_{_{ax}}} \beta -\int_{\gamma_{_{xa}}} \left( \alpha+\dfrac{\beta^2}{\alpha}\right) -2 \int_{\gamma_{_{xa}}} \beta
\end{equation*}
Due to absolute homogeneous property of $F_0,$ above equation takes the form 	
\begin{equation}{\label{gs4.4}}
\begin{split}
\omega_a(x)
&= 2 \int_{\gamma_{_{ax}}} \beta -2 \int_{\gamma_{_{xa}}} \beta\\
&= 4 \int_{\gamma_{_{ax}}} \beta
=-4 \int_{\gamma_{_{xa}}} \beta,
\end{split}
\end{equation}	
where we have used Stoke's theorem for the one-form $\beta$ on the  closed 	domain $D$ with boundary $\partial D := \gamma_{_{ax}} \cup \gamma_{_{xa}}.$ 
 The necessary  and sufficient condition for this to be well defined is that the path integral in R. H. S. of equation (\ref{gs4.4}) is path independent, i.e., one-form $\beta$ is exact.\\
 Further, we have 
\begin{equation}{\label{gs4.5}}
\begin{split}
d_F(x,y)+\omega_a(x)
&= \int_{\gamma_{_{xy}}} \left( \alpha+\dfrac{\beta^2}{\alpha}\right) +2 \int_{\gamma_{_{xy}}} \beta +  2 \int_{\gamma_{_{ax}}} \beta -2 \int_{\gamma_{_{xa}}} \beta\\
&= \int_{\gamma_{_{xy}}} \left( \alpha+\dfrac{\beta^2}{\alpha}\right) -2 \int_{\gamma_{_{xa}}} \beta -2 \int_{\gamma_{_{ya}}} \beta, 
\end{split}
\end{equation}		
where we have again used Stoke's theorem for the one-form $\beta$ on the  closed 	domain with boundary $\gamma_{_{ax}} \cup \gamma_{_{xy}} \cup \gamma_{_{ya}}.$	\\
Similarly, we can find
\begin{equation}{\label{gs4.6}}
\begin{split}
d_F(y,x)+\omega_a(y)
&= \int_{\gamma_{_{yx}}} \left( \alpha+\dfrac{\beta^2}{\alpha}\right) -2 \int_{\gamma_{_{ya}}} \beta -2 \int_{\gamma_{_{xa}}} \beta
\end{split}
\end{equation}		
for the one-form $\beta$ on the  closed 	domain with boundary $\gamma_{_{ay}} \cup \gamma_{_{yx}} \cup \gamma_{_{xa}}.$	\\
From the equations (\ref{gs4.5}) and (\ref{gs4.6}), we conclude that $d_F$ is a weighted quasi-metric with generalized weight $\omega_a.$
\item[(ii)] 
In this case, we have 
\begin{align*}
F(\alpha, \beta)
&=\alpha \cos \left(\dfrac{\beta}{\alpha}\right)+a\beta\\
&=F_0(\alpha, \beta) +a\beta,
\end{align*}
where\ \ $F_0(\alpha, \beta)= \alpha \cos \left(\dfrac{\beta}{\alpha}\right),\ \alpha = \sqrt{a_{ij}(x) y^i y^j} $ is a Riemannian metric,  $\beta=b_i(x) y^i$ is an exact $1-$form and $a$ is a positive non-zero constant.
Clearly $F_0$ is absolute homogeneous, and hence by theorem (\ref{gsthm3.2}), $F$ is with reversible geodesics.\\
Let $\gamma_{_{xy}} \in \Gamma_{_{xy}}$ be an $F-$geodesic, which is in the same time an $F_0-$geodesic, then from equation (\ref{gs4.2}), we have
\begin{equation}{\label{gs4.7}}
\begin{split}
d_F(x,y)
&= \int_{\gamma_{_{xy}}} \left( \alpha \cos \left(\dfrac{\beta}{\alpha}\right)\right) +a \int_{\gamma_{_{xy}}} \beta. 
\end{split}
\end{equation}	
For a fixed point $b \in M,$ let us define a function $\omega_b : M \longrightarrow \mathbb{R}$ given by $ \omega_b(x)=d_F(b,x)-d_F(x,b).$ Then from equation (\ref{gs4.7}), we get 
\begin{equation*}
\omega_{_b}(x)= \int_{\gamma_{_{bx}}} \left(\alpha \cos \left(\dfrac{\beta}{\alpha}\right)\right) +a \int_{\gamma_{_{bx}}} \beta -\int_{\gamma_{_{xb}}} \left(\alpha \cos \left(\dfrac{\beta}{\alpha}\right)\right) -a \int_{\gamma_{_{xb}}} \beta.
\end{equation*}
Due to absolute homogeneous property of $F_0,$ the above equation takes the form 	
\begin{equation}{\label{gs4.8}}
\begin{split}
\omega_{_b}(x)
&= a \int_{\gamma_{_{bx}}} \beta -a \int_{\gamma_{_{xb}}} \beta\\
&= 2a \int_{\gamma_{_{bx}}} \beta 
=-2a \int_{\gamma_{_{xb}}} \beta,
\end{split}
\end{equation}	
where  Stoke's theorem has been used for the one-form $\beta$ on the  closed 	domain $D$ with boundary $\partial D := \gamma_{_{bx}} \cup \gamma_{_{xb}}.$ 
 This is  well defined if and only if the path integral in R. H. S. of equation (\ref{gs4.8}) is path independent, i.e., one-form $\beta$ is exact.\\
Further, we have 
\begin{equation}{\label{gs4.9}}
\begin{split}
d_F(x,y)+\omega_{_b}(x)
&= \int_{\gamma_{_{xy}}} \left(\alpha \cos \left(\dfrac{\beta}{\alpha}\right)\right) +a \int_{\gamma_{_{xy}}} \beta +  a \int_{\gamma_{_{bx}}} \beta -a \int_{\gamma_{_{xb}}} \beta\\
&= \int_{\gamma_{_{xy}}} \left(\alpha \cos \left(\dfrac{\beta}{\alpha}\right)\right) -a \int_{\gamma_{_{xb}}} \beta -a \int_{\gamma_{_{yb}}} \beta, 
\end{split}
\end{equation}		
where we have again used Stoke's theorem for the one-form $\beta$ on the  closed 	domain with boundary $\gamma_{_{bx}} \cup \gamma_{_{xy}} \cup \gamma_{_{yb}}.$	\\
Similarly, we can find
\begin{equation}{\label{gs4.10}}
\begin{split}
d_F(y,x)+\omega_{_b}(y)
&= \int_{\gamma_{_{yx}}} \left(\alpha \cos \left(\dfrac{\beta}{\alpha}\right)\right) -a \int_{\gamma_{_{yb}}} \beta -a \int_{\gamma_{_{xb}}} \beta.
\end{split}
\end{equation}		
From the equations (\ref{gs4.9}) and (\ref{gs4.10}), we conclude that $d_F$ is a weighted quasi-metric with generalized weight $\omega_{_b}.$
\end{enumerate}	
In similar manner, we can easily prove  that, in all remaining cases, the Finsler metric $F$ induces a weighted quasi-metric $d_F$ on $M$.	
\end{proof}
Next, we discuss an interesting geometric property of weighted quasi-metric spaces. \\
Let $(M,F)$ be a Finsler space and $d_F,$ a quasi-metric, induced on $M$ by $F.$ We define
$$ \rho(x,y)=\dfrac{d_F(x,y)+d_F(y,x)}{2}, $$ which is called symmetrization of $d_F.$\\
Next, we recall (\cite{Sha.Bab2017}, \cite{Vit1995}) the following lemma for later use:
\begin{lemma}{\label{gslem4.1}}
	Let $(M,d)$ be any quasi-metric space. Then $d$ is weightable if and only if there exist a function $\omega : M \longrightarrow [0,\infty) $ such that
\begin{equation}{\label{gs4.11}}
d(x,y)=\rho(x,y)+\dfrac{1}{2}\left[ \omega(y)- \omega (x) \right]\ \ \forall \ x,y \in M, 
\end{equation}
where $\rho$ is the symmetrized distance of $d.$ Moreover, \\
we have
\begin{equation}{\label{gs4.12}}
\dfrac{1}{2}\lvert \omega(x)- \omega (y) \rvert \leq \rho(x,y)\ \ \forall \ x,y \in M. 
\end{equation}
\end{lemma}
Below, we give some examples in support of lemma \ref{gslem4.1}.
\begin{example}
Let $(M,F)$ be an $n-$dimensional Finsler space with either of the following metric
\begin{enumerate}
	\item [(i)] $F=\dfrac{(\alpha+\beta)^2}{\alpha},$
	\item [(ii)] $F=\alpha \cos \left(\dfrac{\beta}{\alpha}\right) +a \beta.$
\end{enumerate}
In either case, from theorem \ref{gsthm4.1}, we see that $d_F$ is weighted quasi-metric induced by $F$ on $M$. By using equations (\ref{gs4.3}) and (\ref{gs4.4}), one can easily verify that $d_F$ satisfies equations (\ref{gs4.11}) and (\ref{gs4.12}) of lemma \ref{gslem4.1}. Therefore, $d_F$ is weightable. 
\end{example}
\begin{theorem}
	Let $(M,F)$ be a smooth, simply connected Finsler $n-$manifold and $F$ is either of the following:
	\begin{enumerate}
		\item[(i)]  $F(\alpha, \beta)=\dfrac{(\alpha+\beta)^2}{\alpha},$
		\item[(ii)]  $F(\alpha, \beta)=\alpha \cos \left(\dfrac{\beta}{\alpha}\right)+a\beta,$
		\item[(iii)]  $F(\alpha, \beta)= \alpha \left( \dfrac{\beta}{\alpha} +2\right) ^2,$
		\item[(iv)]  $F(\alpha, \beta)=\alpha \left( \dfrac{\beta^2}{\alpha^2} +\dfrac{\beta}{\alpha}+2\right) ,$
		\item [(v)] $F(\alpha, \beta)=\alpha \left( \dfrac{\beta^2}{\alpha^2} +\dfrac{2 \beta}{\alpha}+2\right).$
	\end{enumerate}	
	Then, in either case, the perimeter length of any geodesic triangle on $M$ does not depend on the orientation, that is,
\begin{equation}{\label{gs4.13}}
d_F(x,y)+d_F(y,z)+d_F(z,x)=d_F(x,z)+d_F(z,y)+d_F(y,x) \ \ \forall\ x,y,z, \ \in \ M.	
\end{equation}
\end{theorem}
\begin{proof}
\begin{enumerate}
	\item[(i)] 	
	We have 
		\begin{align*}
		F(\alpha, \beta)
		&=\dfrac{(\alpha+\beta)^2}{\alpha}\\
		&=F_0(\alpha, \beta)+ 2 \beta,
		\end{align*}
where  $F_0(\alpha, \beta)=\alpha+\dfrac{\beta^2}{\alpha}$ is absolute homogeneous and $d \beta=0.$
\begin{figure}
	\caption{The perimeter length of geodesic triangle is independent of orientation.}
	\centering
	\includegraphics[width=10cm,height=5cm]{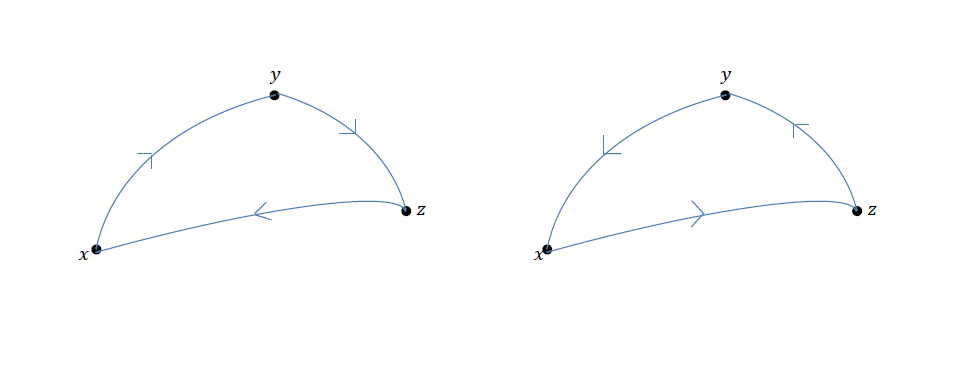} 
\end{figure}
From theorem (\ref{gsthm4.1}), it follows that the quasi-metric is weightable and therefore equation (\ref{gs4.11}) holds good. By using this formula, an elementary computation proves equation (\ref{gs4.13}).\\
Similarly, using theorem (\ref{gsthm4.1}), in all remaining cases, it can be easily verified that  quasi-metrics are weightable and the equation (\ref{gs4.11}) holds good. By using the formula (\ref{gs4.11}), in all remaining cases, an elementary computation gives equation (\ref{gs4.13})  which shows that in either case, the perimeter length of any geodesic triangle on $M$ does not depend on the orientation. 
\end{enumerate}	
\end{proof}

\section{Embeddable quasi-metric spaces with generalized weight}
\begin{definition}
Let $(M,q, \omega)$ and $(M',q', \omega') $ be two weighted quasi metric spaces. A mapping $ \psi :M \longrightarrow M'$ is called  a morphism if $\forall x,y \in M, $ the following properties are satisfied:
\begin{equation}
q(x,y)\geq q'(\psi(x), \psi(y)),
\end{equation}
\begin{equation}
\omega(x) \geq \omega'\left( \psi(x)\right).
\end{equation}
In addition to above two conditions, if $ q(x,y)\leq q'(\psi(x), \psi(y))\ $ holds $ \forall x,y \in M, $ then $\psi$ is called isometric morphism.
\end{definition}
\begin{definition}
Let $(M,q, \omega)$ and $(M',q', \omega') $ be two weighted quasi metric spaces, then the bijection $ \psi :M \longrightarrow M'$ is called  an isomorphism if $\forall x,y \in M, $ the following properties are satisfied:
\begin{equation}
q(x,y) = q'(\psi(x), \psi(y)),
\end{equation}
\begin{equation}
\omega(x) = \omega'\left( \psi(x)\right).
\end{equation}
\end{definition}
\begin{definition}
An embedding of  $(M,q, \omega)$ into $ (N,Q,W)$ is an isomorphism of $(M,q, \omega)$ onto a subspace  $(M',q', \omega') $  of $ (N,Q,W),$ where $M' \subset N, $ and $q', \omega'$ are restrictions of $Q, W $ to $ M' \times  M' $ and $ M'$ respectively.
\end{definition}
\begin{prop}{\label{prop5.1}}
Let $ (S,d)$ be a metric space. Then $ (N,Q,W) $ is a generalized wighted quasi-metric space which we call generalized bundle over $ (S,d),$ where $ N=S\times \mathbb{R},$
$Q: N \times N \longrightarrow \mathbb{R} $ is a generalized quasi-metric defined as
$$ Q\left( (x, \xi),(y, \eta)\right)=d(x,y)+\lambda \left( \eta-\xi\right) ,\ x,y \in S,\ \xi, \eta \in \mathbb{R},  $$
and $ W: N \longrightarrow \mathbb{R} $ is a generalized weight function defined as 
$$ W(x, \xi)= 2 \lambda \xi, \ x \in S,\ \xi \in \mathbb{R}, \ \lambda \geq1 \ \  \text{is an integer.}$$
\end{prop}
\begin{proof}
$\forall x,y,z \in S  $ and $ \xi, \eta, \nu \in \mathbb{R}, $ we have \\
\begin{enumerate}
	\item [(i)] $ Q\left( (x, \xi), (x, \xi) \right) =d(x,x)+\lambda \left( \xi-\xi\right)=0.$
	\item [(ii)]
	\begin{align*}
	Q\left( (x, \xi), (z, \nu) \right) 
	&=d(x,z)+\lambda \left( \nu-\xi\right)\\
	&\leq d(x,y)+\lambda \left( \eta-\xi\right)+ d(y,z)+\lambda \left( \nu-\eta \right)\\
	&=Q\left( (x, \xi),(y, \eta)\right) +Q\left((y, \eta), (z, \nu)\right).
	\end{align*}
	\item [(iii)] Let $ Q\left( (x, \xi),(y, \eta)\right)=Q\left((y, \eta), (x, \xi)\right)=0,$\\
	i.e., $ d(x,y)+\lambda \left( \eta-\xi\right)=d(y,x)+\lambda \left( \xi-\eta\right)=0. $\\
	Since $d(x,y)=d(y,x),$ we have 
	$ \eta=\xi $ \\
	and then $ d(x,y)+\lambda \left( \eta-\xi\right)=0$ gives us $ d(x,y)=0, $\\
	which implies that $x=y.$
	\item [(iv)] 
	\begin{align*}
	Q\left( (x, \xi),(y, \eta)\right)+W(x, \xi)
	&=d(x,y)+\lambda \left( \eta-\xi\right) +2 \lambda \xi\\
	&=d(x,y)+\lambda \left( \eta+\xi\right)\\
	&=d(y,x)+\lambda \left( \xi-\eta\right) +2 \lambda \eta\\
	&=Q\left((y, \eta), (x, \xi)\right)+W(y, \eta).
	\end{align*}
\end{enumerate}
Therefore, $ (N,Q,W) $ is generalized wighted quasi-metric space.
\end{proof}

\begin{theorem}{\label{thm5.2}}
Every weighted quasi-metric space with generalized weight can be embedded in a generalized bundle over a suitable metric space $  (S,d)$.
\end{theorem}
\begin{proof}
Let $ (M, q, \omega)$ be a generalized  weighted quasi-metric space.\\
Let $S=M, d(x,y) =  \rho(x,y) :=\dfrac{q(x,y)+q(y,x)}{2} $ be symmetrization of $q.$\\
Obviously,$ (S,d)$ is a metric space, and then by proposition \ref{prop5.1}, we have a generalized weighted quasi-metric space $ (N, Q, W),$ where 
$ N=S\times \mathbb{R},$\\
$Q: N \times N \longrightarrow \mathbb{R} $ is a generalized quasi-metric defined as
$$ Q\left( (x, \xi),(y, \eta)\right)=d(x,y)+\lambda \left( \eta-\xi\right) ,\ x,y \in S,\ \xi, \eta \in \mathbb{R},  $$
and $ W: N \longrightarrow \mathbb{R} $ is a generalized weight function defined as 
$$ W(x, \xi)= 2 \lambda \xi, \ x \in S,\ \xi \in \mathbb{R}, \ \lambda \geq 1 \ \  \text{is an integer.}$$ 
Define a function $ \psi: M \longrightarrow N $ as
$$ \psi(x)= \left(x, \dfrac{1}{2 \lambda}\ \omega(x) \right),\ x \in M=S.  $$
We claim that $ \psi$ is an embedding, i.e., an isomorphism of $ (M, q, \omega)$ onto a subspace of $ (N, Q, W).$\\
$ \psi $ is one-one.\\
For, let $ x, y \in M $ such that $ \psi(x)=\psi(y) $ \\ 
$ \implies \ \left(x, \dfrac{1}{2 \lambda}\ \omega(x) \right)=\left(y, \dfrac{1}{2 \lambda}\ \omega(y) \right) $\\
$ \implies\ x=y. $\\
Next, $ \psi $ preserves quasi-metric.\\
For, $ \forall\ x, y \in M, $
we have 
\begin{align*}
Q\left( \psi(x), \psi(y)\right)
&=Q\left(\left( x, \dfrac{1}{2 \lambda}\ \omega(x)\right) , \left( y, \dfrac{1}{2 \lambda}\ \omega(y)\right) \right)  \\
&=d(x,y) + \lambda\ \left( \dfrac{\omega(y)- \omega(x)}{2 \lambda} \right) \\
&=\rho(x,y) + \dfrac{\omega(y)- \omega(x)}{2} \\
&= q(x,y). \ \ \ \ \ \ \ \ \ \ \  \left(  \because\ \text{of lemma \ref{gslem4.1}}\right) 
\end{align*}
 Also $ \psi $ preserves weight.\\
 For,  $ \forall\ x \in M, $
 we have 
 \begin{align*}
 W\left( \psi(x)\right)
 &=W\left( x, \dfrac{1}{2 \lambda}\ \omega(x)\right)\\
&= 2 \lambda \left( \dfrac{1}{2 \lambda}\ \omega(x) \right) \\
&=\omega(x).
 \end{align*}
This completes the proof.
\end{proof}
\begin{theorem}{\label{gsthm5.3}}
	Let $(M,F(\alpha, \beta))\ \left(  \beta\  \text{a closed}\  1-\text{form on}\ M\right) $ be a smooth, simply connected Finsler $n-$manifold, where $F$ is either of the following:
	\begin{enumerate}
		\item[(i)]  $F(\alpha, \beta)=\dfrac{(\alpha+\beta)^2}{\alpha},$
		\item[(ii)]  $F(\alpha, \beta)=\alpha \cos \left(\dfrac{\beta}{\alpha}\right)+a\beta,$
		\item[(iii)]  $F(\alpha, \beta)= \alpha \left( \dfrac{\beta}{\alpha} +2\right) ^2,$
		\item[(iv)]  $F(\alpha, \beta)=\alpha \left( \dfrac{\beta^2}{\alpha^2} +\dfrac{\beta}{\alpha}+2\right) ,$
		\item [(v)] $F(\alpha, \beta)=\alpha \left( \dfrac{\beta^2}{\alpha^2} +\dfrac{2 \beta}{\alpha}+2\right).$
	\end{enumerate}	
If $q$ is a weighted quasi-metric on $M$ with generalized weight $ \omega,$
  then the 
space $ (M, q, \omega) $ is embeddable in a generalized bundle over a suitable metric space.
\end{theorem}	
\begin{proof}
Proof directly follows from theorem \ref{thm5.2}.
\end{proof}

\begin{prop}{\label{prop5.4}}
Let $  (S,d )$ be a metric space, and $f$ be a real valued function defined on $S.$ Let $ G_f=\biggl\{ \left( x,f(x) \right) : x\in S  \biggr\}$ be the graph of $f.$ Then $  ( G_f, Q, W)$ is a generalized weighted quasi-metric space,\\
where $ Q: G_f \times G_f \longrightarrow \mathbb{R} $ is a generalized quasi-metric defined as 
$$ Q\left( ( x,f(x)), (y, f(y)) \right)= d(x,y) + \lambda\left(f(y)-f(x) \right), \ x, y \in S   $$  
and $ W : G_f \longrightarrow \mathbb{R} $ is a generalized weight defined as 
$$ W\left(  x,f(x) \right)= 2 \lambda f(x), \ x \in S,\ \lambda \geq 1 \ \text{is an integer.}  $$ 
\end{prop}	
\begin{proof}
$  \forall\ x,y,z \in S,$ we have
\begin{enumerate}
	\item[(i)] $  Q\left( ( x,f(x)), (x, f(x)) \right)= d(x,x) + \lambda\left(f(x)-f(x) \right)=0. $
	\item [(ii)] 
		\begin{align*}
	Q\left( (x, f(x)), (z,f(z)) \right) 
	&=d(x,z)+\lambda \left( f(z)-f(x)\right)\\
	&\leq d(x,y)+\lambda \left( f(y)-f(x)\right)+ d(y,z)+\lambda \left( f(z)-f(y)\right)\\
	&=Q\left( (x,f(x)),(y, f(y))\right) +Q\left((y,f(y)), (z,f(z))\right).
	\end{align*}
	\item [(iii)] Let $Q\left( (x,f(x)),(y, f(y))\right) = Q\left((y,f(y)), (x,f(x))\right)=0.$
This implies\\	$  d(x,y)+\lambda \left( f(y)-f(x)\right)= d(y,x)+\lambda \left( f(x)-f(y)\right)=0. $\\
	Since $ d(x,y)=d(y,x), $ we have $2\lambda \left( f(y)-f(x)\right)=0,$ i.e., $f(y)=f(x)$\\
	and then $d(x,y)+\lambda \left( f(y)-f(x)\right)=0$ gives us $d(x,y)=0$\\
	which implies $ x=y. $
	\item [(iv)] 
	\begin{align*}
	Q\left((x,f(x)),(y, f(y))\right)+W(x,f(x))
	&=d(x,y)+\lambda \left( f(y)-f(x)\right) + 2 \lambda f(x)\\
	&=d(x,y)+\lambda \left( f(y)+f(x)\right)\\
	&=d(y,x)+\lambda \left(f(x)- f(y)\right) + 2 \lambda f(y)\\
	&=Q\left((y, f(y)), (x,f(x))\right)+W(y,f(y)).
	\end{align*}
\end{enumerate}
Therefore, $  ( G_f, Q, W)$ is a generalized weighted quasi-metric space.
\end{proof}	

\begin{definition}
Let $  (S,d )$ be a metric space. Any function $ f: S \longrightarrow \mathbb{R} $ is said to be a $1-$Lipschitz function if it  satisfies
$$ \lvert f(x)-f(y) \rvert \leq d(x,y)\ \forall\ x, y \in S. $$
\end{definition}	

\begin{theorem}{\label{gsthm5.5}}
Let $ (S,d) $ be a metric space and $ f: S \longrightarrow \mathbb{R} $ be a $1-$Lipschitz function. Then $G_f,$ the graph of $f$ is weighted quasi-metric space with generalized weight.
Further, every weighted quasi-metric space with generalized weight can be constructed in similar manner.
\end{theorem}
\begin{proof}
	For first part, take $\lambda=1 $ in proposition \ref{prop5.4}. Since $f$ is a $1-$Lipschitz function, we have $Q\left( ( x,f(x)), (y, f(y)) \right) \geq0.  $\\
	Therefore, $ (G_f,Q,W) $ is a weighted quasi-metric space with generalized weight $ W.$\\
For the second part, let $ (M,q, \omega)$ be a weighted quasi-metric space with generalized weight $ \omega,$\
$ S=M$ and $ d(x,y) = \rho(x,y) :=\dfrac{q(x,y)+q(y,x)}{2} $ be symmetrization of $q,$
then $  (S,d)$ is a metric space.\\
Define $ f: S \longrightarrow \mathbb{R} $ by 
$$ f(x)=\dfrac{1}{2}\ \omega(x)\ \forall \ x \in S,
$$
then$f$ is a $1-$Lipschitz function, as $\forall\ x, y \in S,$ we have
\begin{align*}
\lvert f(x)-f(y) \rvert 
&=\dfrac{1}{2} \lvert \omega(x)-\omega(y) \rvert,\ \\
&\leq \rho(x,y) \ \ \ \ \ \ \left( \text{by lemma \ref{gslem4.1}} \right)  \\
&=d(x,y).
\end{align*}
First part gaurantees that $ \left(G_f, Q, W \right)  $ is a weighted quasi metric space with generalized weight $W,$\\
where  $ G_f=\biggl\{ \left( x,f(x) \right) : x\in S  \biggr\}$ is the graph of $f,$ \\
$ Q: G_f \times G_f \longrightarrow \mathbb{R} $ is a quasi-metric defined as 
$$ Q\left( ( x,f(x)), (y, f(y)) \right)= d(x,y) + f(y)-f(x), \ x, y \in S   $$  
and $ W : G_f \longrightarrow \mathbb{R} $ is a generalized weight defined as 
$$ W\left(  x,f(x) \right)= 2  f(x), \ x \in S.
$$ 
Now, by theorem \ref{thm5.2}, there is an embedding $ \psi: M \longrightarrow N,$ given by 
$ \psi(x)=\left( x, \dfrac{1}{2}\ \omega(x) \right),\  x \in M,$ where $ N=S \times \mathbb{R} $ is a generalized bundle over $ (S,d) $ with a generalized quasi-metric $ Q' : N \times N \longrightarrow \mathbb{R}$ given by 
$$ Q'\left( (x, \xi),(y, \eta)\right)=d(x,y)+ \left( \eta-\xi\right) ,\ x,y \in S,\ \xi, \eta \in \mathbb{R},  $$
and a generalized weight $ W' : N \longrightarrow \mathbb{R} $ given by
$$ W' (x, \xi)= 2  \xi, \ x \in S,\ \xi \in \mathbb{R}.
$$ 
\text{Since}
\begin{align*}
 \psi(M)
&=\left\{ \psi(x)\ \  : x \in M=S \right\}\\
&= \left\{\left( x, f(x) \right)\ \  : x \in S \right\}\\
&= G_f
\end{align*}
and $Q, W$ are restrictions of $Q', W'$ to $ G_f \times G_f $ and $G_f$ respectively,\\
therefore, $ \psi $ is an embedding of $ (M, q, \omega) $ onto $ (G_f, Q, W), $ a subspace of $ (N, Q', W'),$\
i.e., we can identify $ (M, q, \omega) $ with $ (G_f, Q, W).$
\end{proof}

Theorem \ref{gsthm5.6} immediately follows from the theorem \ref{gsthm5.5}.
\begin{theorem}{\label{gsthm5.6}}
	Let $(M,F(\alpha, \beta))\ \left(  \beta\  \text{a closed}\  1-\text{form on}\ M\right) $ be a smooth, simply connected Finsler $n-$manifold, where $F$ is either of the following:
	\begin{enumerate}
		\item[(i)]  $F(\alpha, \beta)=\dfrac{(\alpha+\beta)^2}{\alpha},$
		\item[(ii)]  $F(\alpha, \beta)=\alpha \cos \left(\dfrac{\beta}{\alpha}\right)+a\beta,$
		\item[(iii)]  $F(\alpha, \beta)= \alpha \left( \dfrac{\beta}{\alpha} +2\right) ^2,$
		\item[(iv)]  $F(\alpha, \beta)=\alpha \left( \dfrac{\beta^2}{\alpha^2} +\dfrac{\beta}{\alpha}+2\right) ,$
		\item [(v)] $F(\alpha, \beta)=\alpha \left( \dfrac{\beta^2}{\alpha^2} +\dfrac{2 \beta}{\alpha}+2\right).$
	\end{enumerate}	
	If $q$ is a weighted quasi-metric on $M$ with generalized weight $ \omega,$
	then there is a metric space $ (S,d) $ and a  $1-$Lipschitz function $ f: S \longrightarrow \mathbb{R} $ such that the 
	space $ (M, q, \omega) $ can be viewed as the weighted quasi metric space $ (G_f, Q, W), $ with  generalized weight $W.$
\end{theorem}

\end{document}